\documentclass[12pt]{article}
\usepackage{latexsym}
\usepackage{amsmath}
\newtheorem{theorem}{Theorem}
\newtheorem{lemma}{Lemma}
\newtheorem{proposition}{Proposition}
\newtheorem{corollary}{Corollary}

\newcommand{\qed}{\hfill $\Box$ }

\usepackage[dvips]{graphicx}    
\usepackage{graphics}
\usepackage{epsfig}
\usepackage{subfigure}          

\author{
C. Garcia-Martinez 
\\ \small  Departamento de Ciencias B\'asicas,
\\ \small Universidad Aut\'oma Metropolitana Unidad Azcapotzalco.
\\ \small Av. San Pablo 180, Col. Reynosa Tamaulipas,
\\ \small  C.P 02200 M\'exico D,F.
\\ \small e-mail: cgarcia@correo.azc.uam.mx
\and
F.R. McMorris
\\ \small Department of Applied Mathematics,
Illinois Institute of Technology
\\ \small Chicago, IL 60616 USA
\\ \small and
\\ \small Department of Mathematics,
University of Louisville
\\ \small Louisville, KY 40292 USA
\\ \small e-mail: mcmorris@iit.edu
\and
O. Ortega
\\ \small Department of Mathematics, Harold Washington College
\\ \small Chicago, IL 60601 USA
\\ \small e-mail: oortega@ccc.edu
\and
R.C. Powers
\\ \small Department of Mathematics,
University of Louisville
\\ \small Louisville, KY 40292 USA
\\ \small e-mail: robert.powers@louisville.edu}

\title{Axioms for consensus functions on the n-cube}

\date{\today}

\begin{document}
\maketitle

\newpage

\bigskip
\begin{abstract}
A $p$-value of a sequence $\pi =( x_1,x_2,\ldots,x_k)$ of elements of a
finite  metric space $(X,d)$ is an element $x$ for which
$\sum_{i=1}^{k}d^p(x,x_i)$ is minimum. The $\ell_p$--function with
domain the set of all finite sequences on $X$ and defined by\\
$\ell_p(\pi) = \{x:\ x$ is a $p$-value of $\pi\}$ is called the $\ell_p$--function on $(X,d)$.
The $\ell_1$ and $\ell_2 $ functions are the well studied median and mean functions respectively. In this note, simple characterizations of the $\ell_p$--functions on the n-cube are given. In addition, the center function (using the minimax criterion) is characterized as well as new results proved for the median and anti-median functions.  
\end{abstract}
\bigskip

\medskip

\noindent {\bf Mathematics Subject Classification MSC2010}: 05C90, 90B80

\medskip

\section{Introduction}

A {\it consensus function} (aka {\it location function}) on a finite connected graph $G = (X,E)$ is a mapping $L:X^*\longrightarrow 2^X \backslash \{ \emptyset \}$, where $2^X$ denotes the set of all subsets of $X$, and  $X^* = \bigcup\limits_{k \geq 1} X^k$ with  $X^k = \overbrace{X
\times \cdots \times X}^{k\;\rm times}$. The elements of $X^*$ are called \emph{profiles} and a generic one of {\it length} $k$ is denoted by $\pi = (x_1,x_2, \dots ,x_k)$.  Let $d$ denote the usual geodesic distance, where $d(x,y)$ is the length of a minimum length path joining vertices $x$ and $y$. Suppose the graph $G =(X,E)$ represents the totality of possible locations. Then a profile $\pi = (x_1, \ldots, x_k)$ is formed where $x_i$  represents the best location from the point-of-view of client (voter, customer, user) $i$. A typical approach in location theory is to find those vertices (locations) in $X$ that are ``closest" to the profile $\pi$. There has been much work in this area, ranging from practical computational methods to more theoretical aspects. Since Holzman's paper in 1990  \cite{holz-90}, there have been many axiomatic studies of the procedures themselves which resulted in a much better understanding of the process of location (for a small sample see  \cite{hansen,mcmuvo,mifr-90} and references within). Now suppose the vertex set $X$ is the set of all linear orders (preference ranking) on a given set of alternatives. In this consensus situation, a profile $\pi = (x_1, \ldots, x_k)$ could represent the collection of ballots of the voters labeled by the set $\{1, \ldots, k\}$, i.e., $x_i$ is the preferred ranking of alternatives by voter $i$. Here a closest vertex to $\pi$ would represent the entire group's preferred consensus ranking.  Many references for this classical situation can be found in \cite{damc-03} and other books on voting theory. Another classic situation, and one pertinent to our study, is the process of selecting a committee from a slate of $n$ candidates. Here each of $k$ voters is to nominate a subset of candidates, so a ballot is simply a profile $\pi = (x_1, \ldots, x_k)$ where each $x_i$ is a subset of the candidates (\cite{brams,kilgour}). The vertices of the graph $G$ are the subsets of candidates and the committee consensus function will return one or more subsets closest to the profile.

Four popular measures of the closeness, or remoteness, of a vertex $x$ to a profile $\pi = (x_1, \ldots, x_k)$ are:
\begin{enumerate}
\item The {\it eccentricity} of $x$, $e(x,\pi) = \text{max}\{d(x,x_1),d(x,x_2), \ldots ,d(x,x_k)\} $.
\item The {\it status} of $x$, $S_{\pi}(x) = \sum_{i=1}^kd(x,x_i)$.
\item The {\it square status} of $x$, $SS_{\pi}(x) = \sum_{i=1}^kd^2(x,x_i)$.
\item The {\it $\ell_p$ status} of $x$, $\ell_pS_{\pi}(x) = \sum_{i=1}^kd^p(x,x_i)$.
\end{enumerate}

The consensus functions based on the these measures of remoteness have been defined as follows:\\\\
$(a)$ The \emph{center function}, denoted by \emph{Cen},
is defined by
$$Cen(\pi) = \{x \in X : e(x,\pi)\,\, \text{is minimum}\}.$$
\\
$(b)$\,\, The \emph{median function}, denoted by \emph{Med}, is defined by
$$Med(\pi) = \{x \in X : S_{\pi}(x)\,\, \text{is minimum}\}.$$ \\
$(c)$\,\, The \emph{mean function}, denoted by \emph{Mean}, is defined by $$Mean(\pi)
= \{x \in X : SS_{\pi}(x)\,\, \text{is minimum}\}.$$ \\
$(d)$\,\, The \emph{$\ell_p$-function}, denoted by \emph{$\ell_p$}, is defined by $$\ell_p(\pi)
= \{x \in X : \ell_pS_{\pi}(x)\,\, \text{is minimum}\}.$$ \\
Of course the median and mean functions are special cases of the $\ell_p$--function, but earlier work (\cite{11,mcmuor-12,mn-2011}) shows a striking difference between the case of $p = 1$ and $p > 1$.\\

In this paper we focus on consensus functions on the {\it n-dimensional hypercube} $Q_n = (X,E)$ whose vertex set is $X = \{(w_1, \ldots, w_n): w_i \in \{0,1\}\}$. Of course the natural realization of $Q_n$ is the set of all subsets of an $n$-element set.  Recall that for $u = (u_1, \dots, u_n)$ and $v = (v_1, \ldots, v_n)$ vertices in $Q_n$, $uv$ is an edge of $Q_n$ if and only if $\sum_{i=1}^n|u_i - v_i| = 1$. We set $u \leq v$ if and only if $u_i \leq v_i$ for all $i$. Let $d$ be the usual Hamming distance, where $d(u,v) = \sum_{i=1}^n|u_i - v_i|$, so that $uv$ is an edge if and only if $d(u,v) = 1$.  Let $\oplus$ denote {\it modulo} 2 addition, and $u \oplus v = (u_1 \oplus v_1, \ldots, u_n \oplus v_n)$. For a profile $\pi = (x_1, \dots, x_k)$ and $u \in Q_n$ let $\pi \oplus u = (x_1 \oplus u, \dots, x_k \oplus u)$. Let $\bf{0}$ = $(0, \ldots, 0)$ and $\bf{1}$ = $(1, \dots, 1)$.  Note that $x \oplus x = \bf{0}$ for all $x \in Q_n$. Also it is easy to see that for $x,y$ and $z$ vertices in $Q_n$, $d(x,y) = d(x \oplus z, y \oplus z)$. We set $e_j \in Q_n$ to be the vertex with $0's$ everywhere except $1$ in the $j^{th}$ coordinate. So, for example, in $Q_5$
$$
(0,0,1,1,0) = e_3 \oplus e_4  \mbox{ and } (0,0,1,0,1) \oplus e_3 = (0,0,0,0,1)$$
$$ (0,1,0,1,1) \oplus (0,0,1,1,0) = (0,1,1,0,1) = e_2 \oplus e_3 \oplus e_5.
$$

Let $\langle \pi \rangle$ denote the subgraph induced by the vertices comprising $\pi$. Note that $\langle \pi \oplus v \rangle$ is isomorphic to $\langle \pi \rangle$ for all $v \in Q_n$, and so intuitively $\langle \pi \oplus v \rangle$ is simply a ``translation" of $\langle \pi \rangle$ to another position within $Q_n$. Our goal is to use the particular structure of $Q_n$ to present a very simple unifying approach to give characterizations of $Cen, Med, \mbox{ and } \ell_p$ on these graphs.  Mulder and Novick (\cite{mn-2011,mn-2013}) have given an elegant set of axioms characterizing $Med$ on all median graphs (of which $Q_n$ is a special case) whereas our  axioms are essentially straightforward properties that follow from the definitions. At present the most general graph for which characterizations exist for $Cen$, $Mean$, and $\ell_p$  is a tree (\cite{mcrowa-01, mcmuor-10, mcmuor-12, mupere-08}). An interesting weighted version of $Cen$ on $Q_n$ is studied in \cite{brams}.

We mention that the following results can be framed in the more abstract context of finite Boolean algebras, as it is done in \cite{kp,oscar2014, oscar2015}. We prefer to work in the more specific situation of $n$-cube since properties become quite easy to visualize, and yet we are working without loss of generality since every finite Boolean algebra is isomorphic to an $n$-cube.

\medskip

        \section{The Axioms and Characterizations of $Cen$, $Med$, and $\ell_p\text{-function}$.}

In this section we give two very simple properties that will allow for a general result that can be used to give a new way to view  $Cen, \ell_p$ and $Med$ defined on $Q_n$.  Let $f:X^*\longrightarrow 2^X \backslash \{ \emptyset \}$ be a consensus function on  $Q_n = (X,E)$. Our key axiom for a consensus function $f$ is the following.
\begin{description}
\item[Translation (T)]: For any profile $\pi$ and vertices $u$ and $v$ of $Q_n$,
\[
u \in f(\pi) \mbox{ implies that } u \oplus v \in f(\pi \oplus v)
\]

\end{description}

Note that this is equivalent to $u \in f(\pi) \mbox{ if and only if } u \oplus v \in f(\pi \oplus v)$.

Now let $f$ and $g$ be consensus functions on $Q_n$ and $x_0$ a vertex. Then we say that $f$ and $g$ \emph{agree} at $x_0$ if 
for any profile $\pi$ ,
\[
x_0 \in f(\pi) \mbox{ if and only if } x_0 \in g(\pi).
\]

\begin{theorem} If the consensus functions $f$ and $g$ on $Q_n$ both satisfy $(T)$ and agree at a vertex $x_0$, then $f = g$.
\end{theorem}
\begin{proof}  Let $\pi$ be a profile and $v \in X$. Then there exists $v' \in X$ such that $v \oplus v' = x_0$. Since $f$ satisfies $(T)$, we have
\[
v \in f(\pi) \mbox{ if and only if } v \oplus v' = x_0 \in f(\pi \oplus v').
\]
Because $f$ and $g$ agree at $x_0$,
\[
x_0 \in f(\pi \oplus v') \mbox{ if and only if } x_0 \in g(\pi \oplus v').
\]
Since $g$ satisfies $(T)$, 
\[
v \oplus v' = x_0  \in g(\pi \oplus v') \mbox{ if and only if } v \in g(\pi).
\]
Hence $v \in f(\pi)$ if and only if $v \in g(\pi)$. \qed
\end{proof}

Note that Theorem 1 says that if $f$ and $g$ are consensus functions on $Q_n$ that both satisfy $(T)$, then $f = g$ if the conditions placing {\bf 0} in $f(\pi)$ are the same as the conditions placing {\bf 0} in $g(\pi)$. 

As observed before, $d(x,y) = d(x \oplus z, y \oplus z)$ for $x,y$ and $z$ vertices in $Q_n$. Using this and the definitions it is easy to see that $Cen, Med$ and $\ell_p$ all satisfy $(T)$. Therefore, characterizations will follow once the conditions are obtained for when ${\bf0} \in Cen(\pi)$, ${\bf0} \in Med(\pi)$ and ${\bf0} \in \ell_p(\pi)$.  We present these results in a series of Lemmas and Corollaries. 

Let $u \in Q_n$ and set $\Vert u \Vert = d({\bf 0},u)$, i.e., the number of ones that appear in the representation $u$ as a vertex of $Q_n$. Let $\pi = (x_1,x_2, \ldots . x_k)$ be a profile on $Q_n$. Then $\Vert \pi \Vert$ is defined to be 
\[ \Vert \pi \Vert = max \{\Vert x_1 \Vert, \Vert x_2 \Vert, \ldots ,\Vert x_k \Vert\}.\]

\begin{lemma} Let $Cen$ be the center function on $Q_n$, and $\pi$ a profile. Then
\[ {\bf 0} \in Cen(\pi) \mbox{ if and only if } \Vert \pi \Vert \leq \Vert \pi \oplus u \Vert \,\,  \ for\  all\ u \in Q_n.
\]
\end{lemma}
\begin{proof}
The result is clear because $d(x,y) = d(x \oplus z, y \oplus z)$ in $Q_n$, and $e({\bf 0}, \pi) = \Vert \pi \Vert$ for any profile $\pi$.\qed\\
\end{proof}

\begin{corollary} Let $f$ be a consensus function on $Q_n$. Then $f = Cen$ if and only $f$ satisfies (T) and for every profile $\pi$ and $u \in Q_n$,
\[ {\bf 0} \in f(\pi) \mbox{ if and only if } \Vert \pi \Vert \leq \Vert \pi \oplus u \Vert .
\]
\end{corollary}

Mulder and Novick \cite{mn-2011} give an elegant characterization of $Med$ on $Q_n$, which was extended to all median graphs in \cite{mn-2013}. We will give another characterization using the approach given by Theorem 1. For a profile $\pi = (x_1, \ldots, x_k)$ let $x_i = (x_1^i, \ldots, x_n^i)$. The next result has been noted in \cite{mn-2011}.

\begin{lemma} Let $Med$ be the median function on ${Q_n}$ and $\pi = (x_1, \ldots, x_k)$ a profile.
Then
\[{\bf 0} \in Med(\pi) \mbox{ if and only if }  \sum_{j=1}^kx_i^j \leq \frac{k}{2} \mbox{ for all } i.
\]
\end{lemma}

\begin{corollary}  Let $f$ be a consensus function on $Q_n$. Then $f = Med$ if and only $f$ satisfies (T) and for any profile $\pi = (x_1, \ldots, x_k)$,
\[{\bf 0} \in f(\pi) \mbox{ if and only if }  \sum_{j=1}^kx_i^j \leq \frac{k}{2} \mbox{ for all } i.
\]
\end{corollary}

For the function $\ell_p$ it is easy to see from the definitions that for any profile $\pi$ and $a$ in $Q_n$,

\[ {\bf 0} \in \ell_p(\pi) \mbox{ if and only if } a = {\bf 0} \oplus a \in \ell_p(\pi \oplus a). \]

As in \cite{oscar2015} we consider the $p$\emph{-Characteristic} of  the profile $\pi=(x_1,x_2, \dots,x_{k})$ to be the number
\begin{equation}\label{eq1}
 Char_p(\pi) = \sum_{i = 1}^{n}\Vert x_i\Vert^p. \notag
\end{equation}
Lemma 3.12  in \cite{oscar2015} gives the following proposition.

\begin{lemma} Consider the function $\ell_p$ on ${Q_n}$, and let $\pi = (x_1, \ldots, x_k)$ be a profile. Then
\[{\bf 0} \in \ell_p(\pi) \mbox{ if and only if }\,\, Char_p(\pi) \leq Char_p(\pi \oplus a) \mbox{ for every } a \mbox{ in } Q_n.
\]
\end{lemma}

\begin{corollary}  Let $f$ be a consensus function on $Q_n$. Then $f = \ell_p$ if and only $f$ satisfies (T) and for any profile $\pi = (x_1, \ldots, x_k)$,
\[{\bf 0} \in f(\pi) \mbox{ if and only if } Char_p(\pi) \leq Char_p(\pi \oplus a) \mbox{ for every vertex } a \mbox{ in } Q_n.
\]
\end{corollary}

Here are three other examples of consensus functions that satisfy the Translation property. However it is clear that these functions would not be useful in committee elections or as location functions, for instance.
\begin{description}
\item[Example 1:] Let $f_1$ be the consensus function on $Q_n$ defined by $f(\pi) = \{x_1\}$ for any profile $\pi = (x_1, \ldots, x_k)$. That is, $f_1$ is a standard projection function. Then clearly $f_1$ satisfies (T).
 \item[Example 2:] Let $f_2$ be the consensus function on $Q_n$ defined by $f_2(\pi) = X$ for all profiles $\pi$. That is, $f_2$ is the constant function with output the entire vertex set $X$. Then $f_2$ satisfies (T), and moreover it can be easily shown that it is the only constant function that satisfies (T).
 \item[Example 3:] Let $f_3$ be the consensus function on $Q_n$ defined by $f_3(\pi) = \{\pi\}$ for all $\pi$ where $\{\pi\}$ is the set of vertices appearing in the profile $\pi$. Then clearly $f_3$ satisfies (T).
\end{description} 

The function $f_2$ allows us to see some of the implications of imposing (T). First we need to recall one of the crucial axioms for the characterization of $Med$ (\cite {mcmuro-98,mn-2011,mn-2013}).
\begin{description}
\item[Consistency (C)]: The consensus function $f$ satisfies (C) if, for profiles $\pi_1$ and $\pi_2$; 
\[
f(\pi_1) \cap f(\pi_2) \neq \emptyset \mbox{ implies } f(\pi_1\pi_2) = f(\pi_1) \cap f(\pi_2).
\]

\end{description}

\begin{proposition} A consensus function $f$ on $Q_n$ satisfies (T), (C) and
\[ \bigcap_{x \in X}f(x) \neq \emptyset \]
if and only if $f = f_2$.
\end{proposition}
\begin{proof}
Clearly $f_2$ satisfies the conditions, so now let $f$ be a consensus function that satisfies (T), (C) and the intersection condition. Let $v \in f(x)$ for all $x \in X$. Then since $f$ satisfies (T) we have $v \oplus x \in f(x \oplus x) = f(\bf{0})$ for all $x \in X$. Now let $w$ be an arbitrary vertex. Then $w = v \oplus (v \oplus w) \in f(\bf{0})$ and thus $f({\bf 0})= X$. So if $z$ is any vertex in $X$, $z \oplus x \in f(\bf{0})$ and since $f$ satisfies (T) we have 
\[z = (z \oplus x) \oplus x \in f({\bf 0} \oplus x) = f(x)
 \]
 Therefore $f(x) = X \mbox{ for all } x \in X$, which means that $f(\pi) = X$ for all profiles $\pi$ of length 1. Using (C) and induction we conclude that $f(\pi) = X$ for all profiles $\pi$, i.e., $f = f_2$.
 \qed
\end{proof}

\section{Alternative Characterizations of the Median and Anti-Median Functions on $Q_n$}

For any profile $\pi = (v_1, \ldots, v_k)$ such that
\[ v_i = (x_1^i, \ldots, x_n^i)\in \{0,1\}^n\]
for $i=1, \ldots k$, let $Maj(\pi) = (w_1, \ldots, w_n)$ be the vertex in $X$ such that
\[ w_i =  1 \text{ if and only if } \sum_{j=1}^kx_i^j > \frac{k}{2}\]
for $i=1, \ldots, n$.  We will say that a location function $f$ satisfies condition  $ \mathit{\bf(Maj)}$ if
\[ \mathit{Maj}(\pi) \in f(\pi)\]
for any profile $\pi$. We have previously noted that the median function satisfies (T) and we will show below that, as expected, $Med$ satisfies (Maj). However, there are other location functions that satisfy these two conditions, such as $f_2$ for example. But, arguably, $f_2$ is not a very reasonable method of consensus or location. So our next step is to invoke a condition that restricts the range of a location function.

For any profile $\pi = (v_1, \ldots, v_k)$ such that
\[ v_i = (x_1^i, \ldots, x_n^i)\in \{0,1\}^n\]
for $i=1, \ldots k$, define the {\bf Condorcet score} of $\pi$ to be
\[ Cs(\pi) = |\{i : \sum_{j=1}^k x_i^j = \frac{k}{2}\}|.\]
Observe that if the profile length $k$ is odd, then $Cs(\pi) = 0$. A location function $f$ satisfies {\bf Restricted Range (RR)} if
\[ |f(\pi)| \leq 2^{Cs(\pi)}\]
for any profile $\pi$.

We can now give a completely different characterization of $Med$ from that found in \cite{mn-2011}.

\begin{theorem} Let $f$ be a location function on $Q_n$. Then $f = Med$ if and only if $f$ satisfies (T), (Maj), and (RR).
\end{theorem}

\begin{proof} Assume $f = Med$. We already know that $f$ satisfies (T), so we only need to show that $Med$ satisfies (Maj) and (RR).

We will follow the notation given above. Let $\pi = (v_1, \ldots, v_k)$ be a profile such that
\[ v_i = (x_1^i, \ldots, x_n^i)\in \{0,1\}^n\]
for $i=1, \ldots k$ and let $Maj(\pi) = (w_1, \ldots, w_n) = w$. Now let $a = (y_1, \ldots, y_n) \neq w$ be such that $y_m \neq w_m$ for some $m$. First note that for every $j$, since $w_j$ and $x_j^i$ are equal for at least $\frac{k}{2}$ of the $i's$,
\[ \sum_{i=1}^k|y_j-x_j^i| \geq \sum_{i=1}^k|w_j-x_j^i|.\]
Since
\[ S_{\pi}(a) = \sum_{i=1}^kd(a,v_i)\ \text{where}\ d(a, v_i) = \sum_{j=1}^n|y_j-x_j^i|\]
we have
\[ S_{\pi}(a) = \sum_{i=1}^k \sum_{j=1}^n|y_j-x_j^i| = \sum_{j=1}^n \sum_{i=1}^k|y_j-x_j^i| \geq  \sum_{j=1}^n \sum_{i=1}^k|w_j-x_j^i| = S_{\pi}(w).\]
Therefore $w \in Med(\pi)$ and $f$ satisfies (Maj).

Let $u = (u_1, \ldots, u_n)$ be the vertex in $X$ such that
\[ u_i = 1\ \text{ if and only if }\ \sum_{j=1}^kx_i^j \geq \frac{k}{2}\]
for $i = 1, \ldots, n$. For any vertex $a = (y_1, \ldots, y_n)$ such that $w \leq a \leq u$ and for any $i \in \{1, \ldots, n\}$ such that $\sum_{j=1}^k x_i^j = \frac{k}{2}$ we get that $w_i = 0$, $u_i = 1$, and of course $y_i \in \{0,1\}$. Observe that
\[ \sum_{i=1}^k|y_j-x_j^i| = \frac{k}{2} = \sum_{i=1}^k|w_j-x_j^i|.\]
Since $y_i = w_i$ whenever $\sum_{j=1}^k x_i^j \neq \frac{k}{2}$ it follows that $S_{\pi}(a) = S_{\pi}(w)$ and so $a \in Med(\pi)$. Moreover, if $b = (z_1, \ldots, z_n)$ is vertex in $X$ such that $z_m \neq w_m$ for some $m \in \{1, \ldots, n\}$ where $\sum_{j=1}^k x_m^j \neq \frac{k}{2}$, then
\[ \sum_{i=1}^k|z_m-x_m^i| > \frac{k}{2} > \sum_{i=1}^k|w_m-x_m^i|.\]
In this case, $S_{\pi}(b) > S_{\pi}(w)$ and so $b \not\in Med(\pi)$. It now follows that
\[ Med(\pi) = \{ Maj(\pi) \oplus \sum_{\alpha \in A} i_{\alpha} : A \subseteq S\}\]
where
\[ S = \{\alpha \in \{1, \ldots, n\} : \sum_{j=1}^kx_{\alpha}^j = \frac{k}{2}\}.\]
Therefore, $|Med(\pi)| = 2^{|S|} = 2^{Cs(\pi)}$ and hence $Med$ satisfies (RR).

For the converse, assume that $f$ satisfies (T), (Maj), and (RR). We will show that $f = Med$. Let $\pi = (v_1, \ldots, v_k)$ be a profile. Then, using Theorem 2,
\[ v \in Med (\pi)\ \Leftrightarrow\ {\bf 0} \in Med(\pi \oplus v) \Leftrightarrow\ \sum_{j=1}^ky_i^j \leq \frac{k}{2} \mbox{ for all } i
\]
where $v_j \oplus v = (y_1^j, \ldots, y_n^j)$ for $j=1, \ldots, k$. Observe that $Maj(\pi \oplus v) = {\bf 0}$, and since $f$ satisfies $(Maj)$ it follows that ${\bf 0} \in f(\pi \oplus v)$. Since $f$ satisfies (T) we get
\[ v = {\bf 0} \oplus v \in f(\pi).\]
It now follows that $Med(\pi) \subseteq f(\pi)$ for any profile $\pi$. Therefore,
\[ |Med(\pi)| \leq |f(\pi)|\]
for any profile $\pi$. We know that $Med(\pi) = 2^{Cs(\pi)}$ and, by (RR), that $|f(\pi)| \leq 2^{Cs(\pi)}$ for any profile $\pi$. Hence $f(\pi) = Med(\pi)$ for any profile $\pi$ and we're done. \qed
\end{proof}

The three consensus functions that we have considered all minimize a criterion in order to produce vertices that are close to a given profile of vertices, and as such are useful in location theory. When finding locations to place noxious entities, it is more appropriate to maximize rather than minimize these objective functions, and the resulting ``anti"-functions have also been well-studied. Because we have proved Theorem 2 about the median function, we mention the \emph{anti-median function}, denoted by \emph{AM}, defined by
$$AM(\pi) = \{x \in X : S_{\pi}(x)\,\, \text{is maximum}\}.$$

$AM$ has been characterized on $Q_n$ in \cite{bal}, but we will give an alternate characterization as a corollary to Theorem 2. As before $\pi = (v_1, \ldots, v_k)$ is a profile such that
\[ v_i = (x_1^i, \ldots, x_n^i)\in \{0,1\}^n\]
for $i=1, \ldots k$. Let $Min(\pi) = (m_1, \ldots, m_n)$ be the vertex in $X$ such that
\[ m_i =  1 \text{ if and only if } \sum_{j=1}^kx_i^j < \frac{k}{2}\]
for $i=1, \ldots, n$.  We will say that a location function $f$ satisfies condition {\bf (Min)} if
\[ Min(\pi) \in f(\pi)\]
for any profile $\pi$.  Corollary 1 now follows from the proof of Theorem 2 in the obvious way by reversing the inequalities.

\begin{corollary} Let $f$ be a location function on $Q_n$. Then $f = AM$ if and only if $f$ satisfies (T), (Min), and (RR).
\end{corollary}

\end{document}